\documentclass[12pt]{amsart}
\usepackage{amsfonts,amsthm,amsmath,amssymb}
\usepackage[dvipsnames]{xcolor}
\usepackage{graphicx}
\usepackage[format=plain, font=footnotesize]{caption}
\usepackage{hyperref}
\hypersetup{
  colorlinks=true,
  citecolor=black,
  linkcolor=black,
  urlcolor=black
}

\usepackage{vmargin}
\setpapersize{USletter}
\setmargrb{2.5cm}{2cm}{2.5cm}{2.cm} 
\newtheorem{theorem}{Theorem}
\newtheorem{proposition}[theorem]{Proposition}
\newtheorem{corollary}[theorem]{Corollary}
\newtheorem{lemma}[theorem]{Lemma}

\theoremstyle{definition}

\newtheorem{remark}[theorem]{Remark}

\newcommand{\CC}{{\mathbb C}}
\newcommand{\NN}{{\mathbb N}}

\newcommand{\RR}{{\mathbb R}}
\newcommand{\ZZ}{{\mathbb Z}}

\newcommand{\calA}{{\mathcal A}}
\newcommand{\calI}{{\mathcal I}}
\newcommand{\calJ}{{\mathcal J}}
\newcommand{\calK}{{\mathcal K}}
\newcommand{\calL}{{\mathcal L}}
\newcommand{\calM}{{\mathcal M}}

\newcommand{\calS}{{\mathcal S}}
\newcommand{\calV}{{\mathcal V}}

\DeclareMathOperator{\conv}{{\rm conv}}

\DeclareMathOperator{\EDD}{{\rm EDD}}
\DeclareMathOperator{\MV}{{\rm MV}}

\DeclareMathOperator{\Vol}{{\rm Vol}}

\newcommand{\be}{{\bf e}}
\DeclareMathOperator{\Pyr}{{\rm Pyr}}

\newcommand{\defcolor}[1]{{\color{RoyalBlue}#1}}
\definecolor{OurRed}{rgb}{0.64, 0.30, 0.30}
\newcommand{\demph}[1]{\defcolor{{\sl #1}}}

\title{Euclidean Distance Degree and mixed volume}
\author{P.~Breiding}
\address{Paul Breiding, Max Planck Institute for Mathematics in the Sciences, Inselstra{\ss}e 22, 04103 Leipzig, Germany}
\email{paul.breiding@mis.mpg.de }
\urladdr{http://www.paulbreiding.org}
\author{F.~Sottile}
\address{Frank Sottile, Department of Mathematics,
         Texas A\&M University, College Station, Texas 77843,  USA}
\email{sottile@math.tamu.edu}
\urladdr{http://www.math.tamu.edu/\~{}sottile}
\author{J.~Woodcock}
\address{James Woodcock, Department of Mathematics,
         Texas A\&M University, College Station, Texas 77843,  USA}
\email{jdubbs11@tamu.edu}
\thanks{P.~Breiding funded by the Deutsche Forschungsgemeinschaft (DFG, German Research Foundation) -- Projektnummer 445466444; and funded by the European Research Council
(ERC) under the European Union's Horizon 2020 research and innovation programme (grant agreement No 787840).}
\thanks{Research of Sottile supported by Simons Collaboration Grant for Mathematicians 636314}
\subjclass[2010]{14M25,90C26}
\keywords{Euclidean Distance Degree, Newton Polytopes, Nonlinear Algebra}
\begin{document}

\begin{abstract}
  We initiate a study of the Euclidean Distance Degree in the context of sparse polynomials.
  Specifically, we consider a hypersurface $f=0$ defined by a polynomial~$f$ that is general given its
  support, such that the support contains the origin.
  We show that the Euclidean Distance Degree of $f=0$ equals the mixed volume
  of the Newton polytopes of the associated Lagrange multiplier equations.
We discuss the implication of our result for computational complexity and give a formula for the Euclidean Distance Degree when the Newton
polytope is a rectangular parallelepiped.
\end{abstract}
\maketitle

\section{Introduction}
Let $X\subset \RR^n$ be a real algebraic variety.
For a point $u\in\RR^n\smallsetminus X$,  consider the following  problem:
 \begin{equation}\label{critical_point_problem}
  \text{compute the critical points of }
  d_X\colon X\to \RR, \; x\mapsto \Vert u-x\Vert,
 \end{equation}
where $\|u-x\| = \sqrt{(u-x)^T(u-x)}$ is the Euclidean distance on $\RR^n$.

Seidenberg~\cite{Seidenberg1954}  observed that if $X$ is nonempty, then it contains a solution
to~\eqref{critical_point_problem}. He used this observation in an algorithm for deciding if $X$ is empty.
Hauenstein  \cite{Hauenstein2013} pointed out that solving \eqref{critical_point_problem} provides a point on each connected component of
$X$.
So the solutions to~\eqref{critical_point_problem} are also useful  of $X$
{in learning the number and position of the connected components} of $X$.
From the point of view of optimization, the problem~\eqref{critical_point_problem} is a relaxation of the
optimization problem of finding a point $x\in X$ that minimizes the Euclidean distance to $u$.
A prominent example of this is low-rank matrix approximation, which can be solved by computing the singular value
decomposition.
In general, computing the critical points of the Euclidean distance between~$X$ and $u$ is a difficult task in nonlinear
algebra.

We consider the problem~\eqref{critical_point_problem} when $X\subset \RR^n$ is a \demph{real algebraic hypersurface} in $\RR^n$ defined by
a single real polynomial{,}
\[
 X = \defcolor{\calV_{\RR}(f)} := \{x\in \RR^n \mid f(x)=0\}, \text{ where } f(x) = f(x_1,\ldots,x_n) \in \RR[x_1,\ldots,x_n].
\]
The critical points of the distance function~$d_X$ from~\eqref{critical_point_problem} are called \demph{ED-critical points}.
They can be found by solving the associated
Lagrange multiplier equations. This is a system of polynomial equations defined as follows.

Let us write $\defcolor{\partial_i}$ for the operator of partial differentiation with respect to the variable $x_i$, so that
$\partial_i f:= \frac{\partial f}{\partial x_i}$, and also write
$\defcolor{\nabla f(x)} = (\partial_1 f(x),\ldots, \partial _n f(x))$ for the vector of partial derivatives of $f$ (its \demph{gradient}).
The \demph{Lagrange multiplier equations} are the following system of $n{+}1$ polynomial equations in the $n{+}1$ variables
$(\lambda,x_1,\ldots,x_n)$.
\begin{equation}\label{Lagrange_multiplier_equation}
  \defcolor{\calL_{f,u}(\lambda,x)}\ :=\
     \begin{bmatrix} f(x) \\ \nabla f(x) - \lambda (u-x)\end{bmatrix}\ =\ 0\,,
\end{equation}
where $\lambda$ is an auxiliary variable (the Lagrange multiplier).

We consider the \demph{number of complex solutions} to $\calL_{f,u}(\lambda,x)=0$.
For general $u$, this number is called the \demph{Euclidean Distance Degree} (EDD) \cite{DHOST2016}
of the hypersurface $f=0$:
\begin{equation}\label{EDD_def}
\EDD(f)\ :=\ \text{ number of solutions to } \calL_{f,u}(\lambda,x)=0 \text{ in $\mathbb C^{n+1}$ for general $u$.}
\end{equation}
Here, ``general'' means for all $u$ in the complement of a proper algebraic subvariety of $\RR^n$.
In the following, when referring to $\EDD(f)$ we will simply speak of the EDD of $f$.

Figure~\ref{biquadratic} shows the solutions to $\calL_{f,u}(\lambda,x)=0$ for a biquadratic polynomial $f$.
\begin{figure}[htb]
\includegraphics[height=120pt]{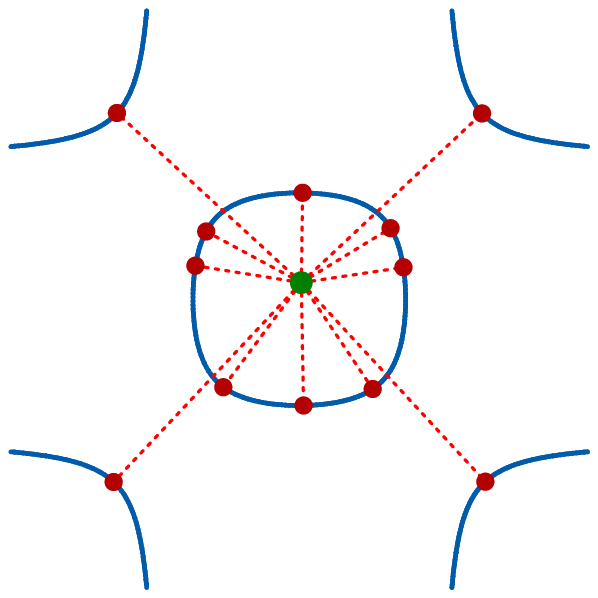}
\caption{The curve $X=\calV_{\RR}(x_1^2 x_2^2 - 3x_1^2 - 3x_2^2+5)\subset \RR^2$ is in blue and $u=(0.025, 0.2)$
is in green.
The 12 red points are the critical points of the distance function $d_X$; that is, they are the $x$-values of the solutions to
  $\calL_{f,u}(\lambda,x)=0$.
  In this example, the Euclidean Distance Degree {of $X$} is 12, so all complex solutions are in fact real.}
\label{biquadratic}
\end{figure}

Determining the Euclidean Distance Degree is of interest in applied algebraic geometry, but also in related areas,
because, as we will discuss in
Section \ref{sec:complexity}, our results on the EDD of~$f$ have implications for the computational complexity of solving the
problem~\eqref{critical_point_problem}.

There is a subtle point about $\EDD(f)$. The definition in~\eqref{EDD_def} does not need us to assume that $\calV_{\RR}(f)$ is
a hypersurface in $\mathbb R^n$.
In fact, $\calV_{\RR}(f)$ can even be empty.
Rather, $\EDD(f)$ is a property of the \demph{complex hypersurface} $\defcolor{X_\CC:=\calV_\CC(f)}$.
We will therefore drop the assumption of $\calV_{\RR}(f)$ being a real hypersurface in the following.
Nevertheless, the reader should keep in mind that for the applications discussed at the beginning of this paper the assumption is needed.
We will come back to those applications only in Sections \ref{sec:real_points} and \ref{certification_discussion}.

In the foundational paper~\cite{DHOST2016}, the Euclidean Distance Degree of $f$ was related to the polar classes of $X_\CC$, and
there are other formulas involving characteristic classes~\cite{AH18} {or Euler characteristic~\cite{MRW}} of~$X_\CC$.
In this paper we give a new formula for the {Euclidean Distance Degree $\EDD(f)$}.

Our main result is Theorem \ref{Th:EDD=MV} in the next section.
We show that, if $f$ is sufficiently general given its support $\mathcal A$ with $0\in\mathcal A$, then $\EDD(f)$ is equal to
the \demph{mixed volume} of the Newton polytopes of~$\calL_{f,u}(\lambda,x)$.
This opens new paths {to compute Euclidean Distance Degree} using tools from convex geometry.
We demonstrate this in Section~\ref{sec:rectangular_parallelepiped} and compute the EDD of a general hypersurface whose Newton polytope is a
rectangular parallelepiped.
{We think it is an interesting problem to relate our mixed volume formula to other formulas involving topological invariants.}

Our proof strategy relies on \demph{Bernstein's Other Theorem} (Proposition \ref{P:Bernstein}) below.
This result gives an effective method for proving that the number of solutions to a system of polynomial equations can be expressed as a
mixed volume.
We hope our work sparks a new line of research that exploits this approach in other applications, not just EDD.

%
\section{Statement of main results}

We give a new formula for the Euclidean Distance Degree that takes into account the monomials in $f$.
In Section~\ref{sec:rectangular_parallelepiped} we work this out in the special case when this Newton polytope is a rectangular
parallelepiped.

Before stating our main results we have to introduce notation:
A vector $a=(a_1,\dotsc,a_n)$  of nonnegative integers is the exponent of a monomial $\defcolor{x^a}:= x_1^{a_1}\cdots x_n^{a_n}$,
and a polynomial $f\in\mathbb C[x_1,\dotsc,x_n]$ is a linear combination of monomials.
The set \defcolor{$\calA$} of exponents of monomials that appear in $f$ is its \demph{support}.
The \demph{Newton polytope} of $f$ is the convex hull of its support.
Given polytopes $Q_1,\dotsc,Q_m$ in $\RR^m$, we write \defcolor{$\MV(Q_1,\dotsc,Q_m)$} for their mixed volume.
This was defined by Minkowski; its definition and properties are explained in~\cite[Sect.~IV.3]{Ewald}, and we revisit them in
Section~\ref{sec:rectangular_parallelepiped}.
Our main result expresses the $\EDD(f)$ in terms of mixed volume.

We denote by \defcolor{$P,P_1,\ldots,P_n$} the Newton polytopes of the Lagrange multiplier equations
$\calL_{f,u}(\lambda,x)$ from \eqref{Lagrange_multiplier_equation}.
That is, $P$ is the Newton polytope of $f$, and $P_i$ is the Newton polytope of~$\partial_i f-\lambda(u_i-x_i)$.
Observe that $P,P_1,\dotsc,P_n$ are polytopes in $\RR^{n+1}$, because $\calL_{f,u}(\lambda,x)$ has~$n+1$ variables
$\lambda,x_1,\ldots,x_n$.

We state our first main result. The proof is given in Section~\ref{S:proof}.

\begin{theorem}\label{Th:EDD=MV}
  If $f$ is a polynomial whose support $\calA$ contains $0$, then
  \[
  \EDD(f)\   \leq\  \MV(P,P_1,\dotsc,P_n)\,,
  \]
  where $P$ is the Newton polytope of $f$ and $P_i$ is the Newton polytope of $\partial_i f-\lambda(u_i-x_i)$ for~$1\leq i\leq n$.
  There is a dense open subset $U$ of polynomials with support $\calA$ such that when $f\in U$ this inequality is an equality
  {and for $u\in\CC^n$ general, all solutions to $\calL_{f,u}$ occur without multiplicity.}
\end{theorem}

{The important point of this theorem is that polynomial systems of the form $\calL_{f,u}$ form a proper subvariety of the set of all
  polynomial systems with the same support---its dimension is approximately $\frac{1}{n}$th of the dimension of the ambient space.
We also remark that the assumption $0\in\calA$ is essential to our proof, and it ensures that $\calV(f)$ is smooth at $0$.}

In the following, we refer to polynomials $f\in U$ as \demph{general given the support $\calA$}.

Since $P,P_1,\dotsc,P_n$ are the Newton polytopes of the entries in $\calL_{f,u}$, Bernstein's Theorem~\cite{bernstein}  implies
the inequality in Theorem \ref{Th:EDD=MV} (commonly known as the \demph{BKK bound}; see also~\cite{EK20}).
Our proof of Theorem~\ref{Th:EDD=MV} appeals to a theorem of Bernstein which gives conditions that imply equality in the BKK bound.
These conditions require the  \demph{facial systems} to be empty.

Our next main result is an application of Theorem~\ref{Th:EDD=MV}. We compute
$\EDD(f)$ when the Newton polytope of {$f$} is the rectangular
parallelepiped
\begin{equation}\label{def_box}
\defcolor{B(a)} := [0,a_1]\times\dotsb\times[0,a_n],
\end{equation}
where $\defcolor{a}:=(a_1,\dotsc,a_n)$ is a list of positive integers.
%
%
For each $1\leq k \leq n$, let
\[
\defcolor{e_k(a)}\ :=\ \sum_{1\leq i_1 <\cdots < i_k\leq n} a_{i_1}\cdots a_{i_k}
\]
be the $k$-th elementary symmetric polynomial in $n$
variables evaluated at $a$. The next theorem is our second main result.

\begin{theorem}
\label{Th:EDD_of_box}
Let $a=(a_1,\dotsc,a_n)$. If $f\in\RR[x_1,\dotsc,x_n]$ has Newton polytope $B(a)$, then
  \[
  \EDD(f)\ \leq\ \sum_{k=1}^n k!\, e_k(a)\,.
  \]
There is a dense open subset $U$ of the space of polynomials with Newton polytope $B(a)$  such that for~$f\in U$, this inequality is an equality.
\end{theorem}

{T}here is a conceptual change when passing from Theorem \ref{Th:EDD=MV} to Theorem \ref{Th:EDD_of_box}.
Theorem~\ref{Th:EDD=MV} is formulated in terms of the support of $f$, whereas Theorem \ref{Th:EDD_of_box}
concerns its Newton polytope.
This is because the equality in {Theorem}~\ref{Th:EDD_of_box} needs the Newton polytope of the partial derivative
$\partial_i f$ to be $B(a_1,\dotsc,a_i{-}1,\dotsc,a_n)$ {for each $1\leq i\leq n$}.

When $n=2$, a polynomial $f$ with Newton polytope the $2\times 2$ square $B(2,2)$ is a biquadratic, and the bound of
Theorem~\ref{Th:EDD_of_box} becomes $2!\cdot 2\cdot 2 + 1!\cdot (2+2)\ =\ 12\,$,
which was the number of critical points found for the biquadratic curve in Figure~\ref{biquadratic}.

\begin{remark}\label{Woodcock's Formula}
  Observe that for $1\leq i_1<\dotsb<i_k\leq n$, if we project $B(a)$ onto the coordinate subspace indexed by
  $i_1,\dotsc,i_k$, we obtain $B(a_{i_1},\dotsc, a_{i_k})$.
Thus the product $a_{i_1}\dotsb a_{i_k}$ is the $k$-dimensional Euclidean volume of this projection and $k!\,a_{i_1}\dotsb a_{i_k}$ is the \demph{normalized volume} of this projection. On the other hand, $e_k(a)=\sum_{1\leq i_1< \cdots <i_k\leq n}a_{i_1}\dotsb a_{i_k}$. This observation implies an appealing interpretation of the formula of Theorem~\ref{Th:EDD_of_box}: It is the sum of the normalized volumes of all coordinate projections of the
  rectangular parallelepiped $B(a)$.\hfill$\diamond$
\end{remark}
\begin{remark}[Complete Intersections]
Experiments with \texttt{HomotopyContinuation.jl} \cite{BT2018} suggest that a similar formula involving mixed volumes should hold for
general complete intersections.
That is, for
\(
    X=\{x\in\RR^n \mid f_1(x) = \cdots = f_k(x)=0\}
\)
such that $\dim X  = n-k$ and~$f_1,\dotsc,f_k$ are general given their Newton polytopes.
The Lagrange multiplier equations~\eqref{Lagrange_multiplier_equation} become $f_1(x)=\cdots=f_k(x)=0$ and $J\lambda -(u-x)=0$,
where $\lambda = (\lambda_1,\ldots,\lambda_k)$ is now a vector of variables, and $J=(\nabla f_1,\dotsc,\nabla f_k)$ is the
$n\times k$ Jacobian matrix.

We leave this general case of $k>1$ for further research.\hfill$\diamond$
\end{remark}

\subsection{Acknowledgments}
The first and the second author would like to thank the organizers of the Thematic Einstein Semester on
Algebraic Geometry: ``Varieties, Polyhedra, Computation'' in the Berlin Mathematics Research Center MATH+.
This thematic semester included a research retreat where the first and the second author first discussed the relation
between Euclidean Distance Degree and mixed volume, inspired by results in~\cite{DKS}.
The first author would like to thank Sascha Timme for discussing the ideas in Section \ref{certification_discussion}.
\subsection{Outline}
In Section~\ref{sec:complexity} we explain implications of Theorem~\ref{Th:EDD=MV} for computational complexity
in the context of using the \demph{polyhedral homotopy} for solving the Lagrange multiplier equations $\calL_{f,u}=0$
for the problem~\eqref{critical_point_problem}.
In Section~\ref{S:proof}, we explain Bernstein's conditions and give a proof of Theorem~\ref{Th:EDD=MV}.
The proof relies on a lemma asserting
that the facial systems of~$\calL_{f,u}$ are empty. Section~\ref{S:EDD=MV} is devoted to proving this lemma.
The arguments that are used in this proof are explained on an example at the end of Section~\ref{S:proof}.
We conclude in Section~\ref{sec:rectangular_parallelepiped} with a proof of Theorem \ref{Th:EDD_of_box}.
%
\section{Implications for computational complexity}\label{sec:complexity}

We discuss the implications of Theorem~\ref{Th:EDD=MV} for the computational complexity of computing critical points of the Euclidean
distance~\eqref{critical_point_problem}.

\subsection{Polyhedral homotopy is optimal for EDD}

Polynomial homotopy continuation is an algorithmic framework for numerically solving polynomial equations which builds upon the following
basic idea:
Consider the system of $m$ polynomials $F(x)=(f_1(x),\ldots, f_m(x))= 0$ in variables $x=(x_1,\ldots,x_m)$.
The approach to solve $F(x)=0$ is to generate another system~$G(x)$ (the \demph{start system}) whose zeros are known.
Then, $F(x)$ and $G(x)$ are joined by a \demph{homotopy}, which is a system $H(x,t)$ of polynomials in $m{+}1$ variables with $H(x,1)=G(x)$ and~$H(x,0)=F(x)$.
Differentiating $H(x,t)=0$ with respect to $t$ leads to an ordinary differential equation called \demph{Davidenko equation}.
The ODE is solved by standard numerical continuation methods with initial values the zeros of $G(x)$. This process is usually called \demph{path-tracking} and \demph{continuation}.
For details see \cite{Sommese:Wampler:2005}.

One instance of this framework is the \demph{polyhedral homotopy} of Huber and Sturmfels \cite{HS1995}.
It provides a start system $G(x)$ for polynomial homotopy continuation and a homotopy~$H(x,t)$ such that the following holds: Let
$Q_1,\ldots,Q_m$ be the Newton polytopes of~$F(x)$.
Then, for all~$t\in(0,1]$ the system of polynomials $H(x,t)$ has $\MV(Q_1,\ldots,Q_m)$ isolated zeros (at $t=0$ this can fail,
because the input $F(x)=H(x,0)$  may have fewer than $\MV(Q_1,\ldots,Q_m)$ isolated zeroes).
Polyhedral homotopy is implemented in many polynomial homotopy continuation software {packages};
for instance in \texttt{HomotopyContinuation.jl} \cite{BT2018}, \texttt{HOM4PS} \cite{Lee:Li:Tsai:2008}, \texttt{PHCPack}
\cite{Verschelde:PHCpack}.

Theorem \ref{Th:EDD=MV} implies that {the} polyhedral homotopy is optimal for computing ED-critical points in the following sense:
If we assume that the continuation of zeroes has unit cost, then the complexity of solving a system of polynomial equations $F(x)=0$ by
polynomial homotopy continuation is determined by the number of paths that have to be tracked.
{This number is at least as large as the number of solutions to $F(x)=0$ that are computed.}
We say that a homotopy is \demph{optimal} if the following three properties hold:
(1) the start system $G(x)$  has as many zeros as the input $F(x)$;
(2) all continuation paths end in a zero of $F(x)$; and~(3) for  {every zero} of $F(x)$ there is a continuation path which converges to it.
In an optimal homotopy no continuation paths have to be sorted out.
T{hat is, t}he number of paths {which need to be tracked is optimal.}

We now have the following consequence of Theorem \ref{Th:EDD=MV}{, as $\calL_{f,u}=0$ has $\MV(P,P_1,\dotsc,P_n)$ isolated solutions}.

\begin{corollary}\label{cor_complexity}
If $f$ is general given its support $\calA$ with $0\in \calA$, polyhedral homotopy is optimal for solving $\calL_{f,u}=0$.
\end{corollary}
Corollary \ref{cor_complexity} is {is an instance} of a structured problem for which we have an optimal homotopy available.

In our definition of optimal homotopy we ignored the computational complexity of
path-tracking in polyhedral homotopy. We want to emphasize that this is an important part of contemporary research.
We refer to Malajovich's work \cite{Malajovich2017, Malajovich2019, Malajovich2020}.

\subsection{Computing real points on real algebraic sets}\label{sec:real_points}

Hauenstein  \cite{Hauenstein2013} observed that solving the Lagrange multiplier equations $\calL_{f,u}=0$ gives at least one point on
each connected component of the real algebraic set $X=\calV_{\RR}{(f)}$.
%
%
Indeed, every real solution to~$\calL_{f,u}=0$ corresponds to a critical point of the distance function from
\eqref{critical_point_problem}.
Every connected component of $X$ contains at least one such critical point.

Corollary \ref{cor_complexity} shows that polyhedral homotopy provides an optimal start system for Hauenstein's approach.
Specifically, Corollary \ref{cor_complexity} implies that when using polyhedral homotopy in the algorithm in
\cite[Section  2.1]{Hauenstein2013}, one does not need to distinguish between the sets $E_1$ (=~continuation paths which converge to a
solution to $\calL_{f,u}=0$) and $E$ (=~continuation paths which diverge).
This reduces the complexity of Hauenstein's algorithm, who puts his work in the context of complexity in real algebraic geometry
\cite{ARS2002, BPR1996, RRS2000, Seidenberg1954}.

\subsection{Certification of ED-critical points}\label{certification_discussion}

We consider \demph{a posteriori certification} for polynomial homotopy continuation:
Zeros are certified after and not during the (inexact) numerical continuation.
Implementations using exact arithmetic \cite{HS2012, LeeM2}  or interval arithmetic \cite{BRT2020, LeeM2, Rump1999} are available.
In particular, box interval arithmetic in $\mathbb C^n$ is powerful in combination with our results.
We explain this.

Box interval arithmetic in the complex numbers is arithmetic with intervals of the form
$\{x+{\sqrt{-1}}y \mid x_1\leq x\leq x_2, \, y_1\leq y\leq y_2\}$ for $x_1,x_2,y_1,y_2\in\mathbb R$.
Box interval arithmetic in $\mathbb C^n$ uses products of such intervals.
{By Theorem~\ref{Th:EDD=MV}, if $f$ is general given its support and $u\in\CC^n$ is general, then  $\calL_{f,u}$ has exactly
  $\MV(P,P_1,\dotsc,P_n)$ solutions.
  Therefore, if we compute $\MV(P,P_1,\dotsc,P_n)$ numerical approximations to solutions, and then certify that each}
corresponds to a true zero, and if we can certify that those true zeros are pairwise distinct, we have provably obtained all zeros
of $\calL_{f,u}$.
Furthermore, if we compute box intervals in~$\mathbb C^{n+1}$ which provably contain the zeros of $\calL_{f,u}$, {then} we can use those
intervals to certify whether a zero is real (see \cite[Lemma 4.8]{BRT2020}) or whether it is not real (by checking if the intervals
intersect the real line; this is a property of box intervals).

If it is possible to classify reality for all zeros, we can take  {a set of} intervals $\{r_1,\ldots,r_k\}$ of $\RR^n$ which
contain the real critical points of the distance function $d_X$ from~\eqref{critical_point_problem}.
The $r_j$ are obtained from the coordinate projection $(\lambda,x)\mapsto x$ of the intervals containing the real zeros of $\calL_{f,u}$.
Setting $d_j:=\{d_X(s)\mid s\in r_j\}$ gives a set of intervals  $\{d_1,\ldots,d_k\}$ of $\RR$.
If there exists $d_i$ such that~$d_i\cap d_j =\emptyset $ and $\min d_i < \min d_j$  for all $i\neq j$, then this is a proof that the
minimal value of~$d_X$ is contained in $d_i$ and that the minimizer for $d_X$ is contained in $r_i$.
%

\section{Bernstein's Theorem}\label{S:proof}
The relation between number of solutions to a polynomial system and mixed volume is given by
Bernstein's Theorem~\cite{bernstein}.

Let $g_1,\dotsc,g_m\in\CC[x_1,\dotsc,x_m]$ be $m$ polynomials with Newton polytopes $Q_1,\dotsc,Q_m$.
Let~\defcolor{$(\CC^\times)^m$} be the complex torus of $m$-tuples of nonzero complex numbers and \defcolor{$\#\calV_{\CC^\times}(g_1,\dotsc,g_m)$} be the number of isolated solutions to $g_1=\dotsb=g_m=0$ in~$(\CC^\times)^m$, counted
by their algebraic multiplicities.
Bernstein's Theorem~\cite{bernstein} asserts that
 \begin{equation}\label{Eq:BKK_Bound}
   \#\calV_{\CC^\times}(g_1,\dotsc,g_m)\ \leq\ \MV(Q_1,\dotsc,Q_m)\,,
 \end{equation}
and the inequality becomes an equality when each $g_i$ is general given its support.
The restriction of the domain to $(\CC^\times)^m$ is because Bernstein's Theorem
concerns Laurent polynomials, in which the exponents in a monomial are allowed to be negative.

An important special case of Bernstein's Theorem was proven earlier by Kushnirenko.
Suppose that the polynomials $g_1,\dotsc,g_m$ all have the same Newton polytope. This means that~$Q_1=\dotsb=Q_m$.
We write $Q$ for this single polytope.
Then, the mixed volume in~\eqref{Eq:BKK_Bound} becomes
 $ \MV(Q_1,\dotsc,Q_m) = m! \Vol(Q)$,
where \defcolor{$\Vol(Q)$} is the $m$-dimensional Euclidean volume of $Q$.
Kushnirenko's Theorem~\cite{Kushnirenko} states that if $g_1,\dotsc,g_m$ are general polynomials with Newton polytope $Q$, then
\[
   \#\calV_{\CC^\times}(g_1,\dotsc,g_m)\ =\ m! \Vol(Q)\,.
\]
That the mixed volume becomes the normalized Euclidean volume when the polytopes are equal is
one of three properties which characterize mixed volume, the others being symmetry and multiadditivity.
This is explained in~\cite[Sect.~IV.3]{Ewald} and recalled in Section~\ref{sec:rectangular_parallelepiped}.

The inequality~\eqref{Eq:BKK_Bound} is called the BKK bound~\cite{BKK}.
The key step in proving it is what we call
\demph{Bernstein's Other Theorem}.
This {\it a posteriori} gives the condition under which the inequality~\eqref{Eq:BKK_Bound} is strict (equivalently, when it is an
equality).
We explain that.

Let $g\in\CC[x_1,\dotsc,x_m]$ be a polynomial with support $\calA\subset\ZZ^m$, so that
\[
   g\ =\ \sum_{a\in\calA} c_a x^a\  \qquad (c_a\in\CC)\,.
\]
For $w\in\ZZ^m$, define \defcolor{$h_w(\calA)$} to be the minimum value of the linear function $x\mapsto w\cdot x$ on the set $\calA$
and write \defcolor{$\calA_w$} for the subset of $\calA$ on which this minimum occurs.
This is the \demph{face} of $\calA$ exposed by $w$.
We write
\begin{equation}\label{def_g_w}
  \defcolor{g_w}\ :=\ \sum_{a\in \calA_w}  c_a z^a\,,
\end{equation}
for the restriction of $g$ to $\calA_w$.
For $w\in\ZZ^{m}$ and a system $G=(g_1,\dotsc,g_m)$ of $m$ polynomials, the \demph{facial system}
is $\defcolor{G_w}:=((g_1)_{w},\dotsc,(g_m)_{w})$.

We state Bernstein's Other Theorem~\cite[Theorem B]{bernstein}.

\begin{proposition}[Bernstein's Other Theorem]\label{P:Bernstein}
  Let $G=(g_1,\dotsc,g_m)$ be a system of Laurent polynomials in variables $x_1,\dotsc,x_m$.
  For each $1\leq i \leq m$, let $\calA_i$ be the support of $g_i$ and $Q_i=\conv(\calA_i)$ its Newton polytope.
  Then
  \[
  \#\calV_{\CC^\times}(g_1,\dotsc,g_m)\ <\ \MV(Q_1,\dotsc,Q_m)
  \]
  if and only if there is $0\neq w\in\ZZ^{m}$ such that the facial system $G_w$ has a solution in $(\CC^\times)^{m}$.
  Otherwise, $ \#\calV_{\CC^\times}(g_1,\dotsc,g_m)$ is equal to $\MV(Q_1,\dotsc,Q_m)$
\end{proposition}

While this statement is similar to Bernstein's formulation, we use its contrapositive, that
the number of solutions equals the mixed volume when no facial system has a solution.
We use Bernstein's Other Theorem when $G = \calL_{f,u}$ and $m=n{+}1$.
For this, we must show that for a general polynomial $f$ with support $\calA\subset\NN^n$, all the solutions to $\calL_{f,u}=0$ lie
in $(\CC^\times)^{n+1}$ and no facial system  $(\calL_{f,u})_w=0$ for $0\neq w\in\ZZ^{n+1}$ has a solution in $(\CC^\times)^{n+1}$.
The later is given by the next theorem which is proved in Section \ref{S:EDD=MV}.

\begin{theorem}\label{Thm:EmptyFaces}
  Suppose that $f$ is general given its support $\calA$, that $0\in\calA$, and that $u\in\CC^n$ is general.
  For any nonzero $w\in\ZZ^{n+1}$,  the facial system $(\calL_{f,u})_w$ has no solutions in~$(\CC^\times)^{n+1}$.
\end{theorem}
Using this theorem we can now prove Theorem \ref{Th:EDD=MV}.
\begin{proof}[Proof of Theorem \ref{Th:EDD=MV}]
Suppose that {a polynomial} $f(x) \in\CC[x_1,\dotsc,x_n]$ is general given its support $\calA$ and
that~$0\in\calA$.
We may also suppose that $u\in\CC^n\smallsetminus\calV_{\CC}(f)$ is general.
By Theorem~\ref{Thm:EmptyFaces}, no facial system $(\calL_{f,u})_w$ has a solution.
By Bernstein's Other Theorem, the Lagrange multiplier equations $\calL_{f,u}=0$ have $\MV(P,P_1,\dotsc,P_n)$ solutions in
$(\CC^\times)^{n+1}$.
It remains to show that there are no other solutions to the Lagrange multiplier equations.

For this, we use standard dimension arguments, such as \cite[Theorem 11.12]{Harris}, and freely invoke the generality
of $f$.
Consider the~\demph{incidence variety}
\[
  \defcolor{\calS_f}\ :=\ \{(u,\lambda,x)\in\CC^n_u\times \CC_\lambda \times \CC^n_x \mid \calL_{f,u}(\lambda,x)=0\}\,,
\]
which is an affine variety.
As $f=0$ is an equation in $\calL_{f,u}=0$, this is a subvariety of $\CC^n_u\times\CC_\lambda\times X_\CC$, where $X_\CC$ is the
complex hypersurface $X_\CC=\calV_\CC(f)$.

Write $\pi$ for the projection of $\calS_f$ to $X_\CC$ and let $x\in X_\CC$.
The fiber $\pi^{-1}(x)$ over $x$ is
\[
   \{ (u,\lambda)\in \CC^n_u\times \CC_\lambda \mid \nabla f(x)=\lambda(u-x)\}\,.
\]
Let $(u,\lambda)\in\pi^{-1}(f)$.
As $f$ is general, $X_\CC$ is smooth, so that $\nabla f(x)\neq 0$ and we see that~$\lambda\neq 0$ and $u\neq x$.
Thus $u = x + \tfrac{1}{\lambda}\nabla f(x)$.
This identifies the fiber $\pi^{-1}(x)$ with $\CC^\times_\lambda$, proving that~$\calS_f\to X_\CC$ is a $\CC^\times$-bundle, and thus is
irreducible of dimension $n$.

{The projection of $\calS_f$ to $\CC^n_u$ is dominant, and therefore Bertini's Theorem implies that the general fiber is
  zero-dimensional and smooth.
That is, for $u\in\CC^n_u$ general, $\calL_{f,u}=0$ has finitely many solutions and each has multiplicity 1.}

Let $\defcolor{Z}\subset X_\CC$ be the set of points of $X_\CC$ that do not lie in $(\CC^\times)^n$ and hence lie on some coordinate plane.
As $f$ is irreducible and $f(0)\neq 0$, we see that $Z$ has dimension $n{-}2$, and its inverse image $\pi^{-1}(Z)$ in $\calS_f$ has
dimension $n{-}1$.
The image $W$ of $\pi^{-1}(Z)$ under the projection to $\CC^n_u$ consists of those points $u\in \CC^n_u$ which have a solution $(x,\lambda)$
to $\calL_{f,u}(\lambda,x)=0$ with $x\not\in(\CC^\times)^n$.
Since $W$ has dimension at most $n{-}1$, this shows that for general $u$ all solutions to $\calL_{f,u}(\lambda,x)=0$
lie in $(\CC^\times)^{n+1}$ (we already showed that $\lambda\neq 0$).

This  completes the proof of Theorem~\ref{Th:EDD=MV}.
\end{proof}

\subsection{Application of Bernstein's other theorem}
To illustrate Theorem \ref{Thm:EmptyFaces}, let us consider two facial systems of the Lagrange multiplier
equations in an example.

Let $\defcolor{\partial_i\calA}$ be the support of  $\partial_i f$. It depends upon the
support $\calA$ of $f$ and the index $i$ in the following way.
Let $\defcolor{\be_i}:=(0,\ldots,0,1,0,\ldots,0)$ be the $i$th standard basis vector ($1$ is in position~$i$).
To obtain $\partial_i\calA$ from $\calA\subset\NN^n$, first remove all points $a\in\calA$ with $a_i=0$, then shift the remaining
points by $-\be_i$. The support of $\partial_i f-\lambda(u_i-x_i)$ is obtained by adding~$\be_{0}$ and~$\be_i+\be_{0}$ to
$\defcolor{\partial_i\calA}$.
{(As usual, we identify $\NN^n$ with $\{0\}\times\NN^n\subset\NN^{n+1}$.)}
Throughout the paper we associate to $\lambda$ the \demph{exponent with index $0$}.

Consider the polynomial in two variables,
\[
  f\ =\ c_{00}+c_{10}x_1+c_{01}x_2+c_{11}x_1x_2 + c_{21}x_1^2x_2\,.
\]
Its support is $\defcolor{\calA} = \{(0,0), (0,1), (1,1), (2,1), (1,0)\}$ and its Newton polytope is $P=\conv(\calA)$, which is a trapezoid.
Figure \ref{fig1} shows the Newton polytope $P$ along with the Newton polytopes of $\partial_1 f-\lambda(u_1-x_1)$ and
$\partial_2 f-\lambda(u_2-x_2)$.
These are polytopes in $\RR^3$; we plot the exponents of the Lagrange multiplier $\lambda$ in the (third) vertical direction in
Figure~\ref{fig1}.
\begin{figure}[htb]
\centering
   \begin{picture}(144,100)(-35,-8)
    \put(  0, 0){\includegraphics{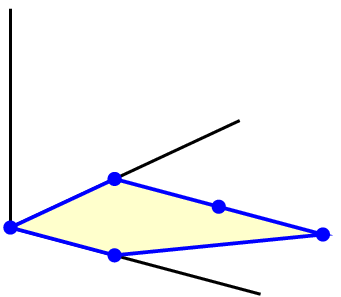}}
    \put(  6,41){\small$(0,1,0)$}    \put(53,33){\small$(1,1,0)$}
    \put(-35,18){\small$(0,0,0)$}
    \put(  5, 0){\small$(1,0,0)$}    \put(75,6){\small$(2,1,0)$}
    \put( 28,20){\small$P$}
  \end{picture}
  \quad
   \begin{picture}(130,100)(-35,-8)
    \put(  0, 0){\includegraphics{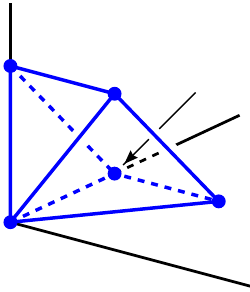}}
    \put(-35,63){\small$(0,0,1)$}      \put(15,66){\small$(1,0,1)$}
    \put( 58,60){\small$(0,1,0)$}
    \put( 50,15){\small$(1,1,0)$}    \put(-35,18){\small$(0,0,0)$}
    \put(  1, 5){\small$P_1$}
  \end{picture}
  \quad
   \begin{picture}(115,100)(-35,-8)
    \put( 0,  0){\includegraphics{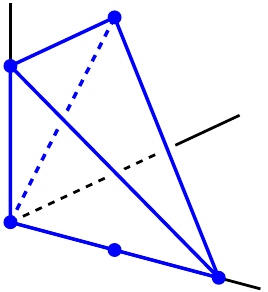}}
    \put(-35,63){\small$(0,0,1)$}    \put( 16,84){\small$(0,1,1)$}
    \put(-35,18){\small$(0,0,0)$}    \put( -1, 2){\small$(1,0,0)$}
    \put( 40,-8){\small$(2,0,0)$}
    \put( 58,25){\small$P_2$}
  \end{picture}
   \caption{The three Newton polytopes of $\calL_{f,u}$ for $f=c_{00}+c_{10}x_1+c_{01}x_2+c_{11}x_1x_2 + c_{21}x_1^2x_2$.}
   \label{fig1}
\end{figure}

The faces exposed by  $w=(0,1,0)$ are shown in red in Figure \ref{fig2}.
\begin{figure}[htb]
\centering
   \begin{picture}(144,100)(-35,-8)
    \put(  0, 0){\includegraphics{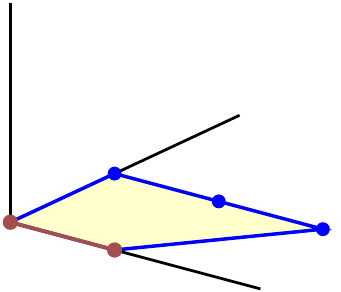}}
    \put(  6,41){\small$(0,1,0)$}    \put(53,33){\small$(1,1,0)$}
    \put(-35,18){\color{OurRed}\small$(0,0,0)$}
    \put(  5, 0){\color{OurRed}\small$(1,0,0)$}    \put(75,6){\small$(2,1,0)$}
  \end{picture}
  \quad
   \begin{picture}(130,100)(-35,-8)
    \put(  0, 0){\includegraphics{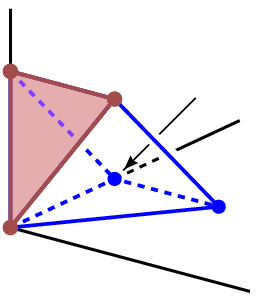}}
    \put(-35,63){\color{OurRed}\small$(0,0,1)$}      \put(15,66){\color{OurRed}\small$(1,0,1)$}
    \put( 58,60){\small$(0,1,0)$}
    \put( 50,15){\small$(1,1,0)$}    \put(-35,18){\color{OurRed}\small$(0,0,0)$}
  \end{picture}
  \quad
   \begin{picture}(115,100)(-35,-8)
    \put( 0,  0){\includegraphics{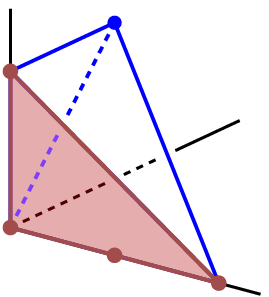}}
    \put(-35,63){\color{OurRed}\small$(0,0,1)$}    \put( 16,84){\small$(0,1,1)$}
    \put(-35,18){\color{OurRed}\small$(0,0,0)$}    \put( -1, 2){\color{OurRed}\small$(1,0,0)$}
    \put( 40,-8){\color{OurRed}\small$(2,0,0)$}
  \end{picture}
   \caption{The faces $\calA_w$, $(\calA_1)_w$ and $(\calA_2)_w$ for $w=(0,1,0)$ are shown in red.}
   \label{fig2}
\end{figure}

The corresponding facial system is
\[
   (\calL_{f,u})_w\ =\
         \begin{bmatrix}
               c_{00}+c_{10}x_1\\
               c_{10}-\lambda(u_1-x_1)\\
               c_{01}+c_{11}x_1+c_{21}x_1^2-\lambda u_2
          \end{bmatrix}\ .
\]
Let us solve $(\calL_{f,u})_w=0$.
We solve the first equation for $x_1$, and then substitute that into the second equation and solve it for $\lambda$ to obtain
\[
   x_1\ =\ -\frac{c_{00}}{c_{10}}\qquad\mbox{and}\qquad
   \lambda\ =\ \frac{c_{10}}{u_1-x_1}\ =\ \frac{c_{10}^2}{c_{10}u_1 + c_{00}}\ .
\]
Substituting these into the third equation and clearing denominators gives the equation
\[
0\ =\ (c_{10}u_1 + c_{00})(c_{10}^3 - c_{11} c_{10} c_{00} + c_{00}^2 c_{21}) \ -\ c_{10}^4 u_2
\]
which does not hold for $f,u$ general.
The proof of Theorem \ref{Thm:EmptyFaces} is divided in three cases and one involves such \demph{triangular systems}, which
are independent of some of the variables.

The faces exposed by $w=(0,-1,1)$ are shown in red in Figure \ref{fig3}.
\begin{figure}[htb]
\centering
   \begin{picture}(144,100)(-35,-8)
    \put(  0, 0){\includegraphics{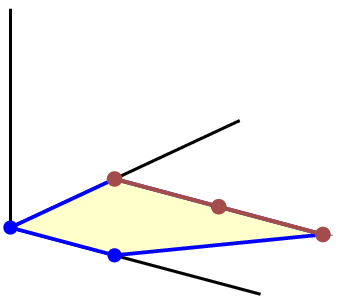}}
    \put(  6,41){\color{OurRed}\small$(0,1,0)$}    \put(53,33){\color{OurRed}\small$(1,1,0)$}
    \put(-35,18){\small$(0,0,0)$}
    \put(  5, 0){\small$(1,0,0)$}    \put(75,6){\color{OurRed}\small$(2,1,0)$}
  \end{picture}
  \quad
   \begin{picture}(130,100)(-35,-8)
    \put(  0, 0){\includegraphics{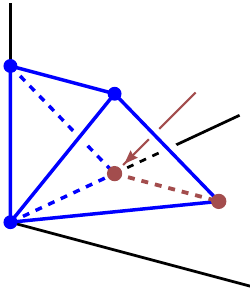}}
    \put(-35,63){\small$(0,0,1)$}      \put(15,66){\small$(1,0,1)$}
    \put( 58,60){\color{OurRed}\small$(0,1,0)$}
    \put( 50,15){\color{OurRed}\small$(1,1,0)$}    \put(-35,18){\small$(0,0,0)$}
  \end{picture}
  \quad
   \begin{picture}(115,100)(-35,-8)
    \put( 0,  0){\includegraphics{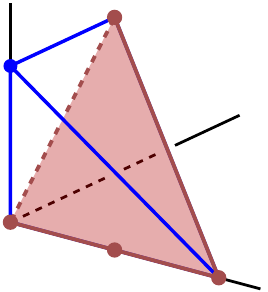}}
    \put(-35,63){\small$(0,0,1)$}    \put( 16,84){\color{OurRed}\small$(0,1,1)$}
    \put(-35,18){\color{OurRed}\small$(0,0,0)$}    \put( -1, 2){\color{OurRed}\small$(1,0,0)$}
    \put( 40,-8){\color{OurRed}\small$(2,0,0)$}
  \end{picture}
  \caption{The faces $\calA_w$, $(\calA_1)_w$ and $(\calA_2)_w$ for $w=(0,-1,1)$ are shown in red.}
  \label{fig3}
\end{figure}
The corresponding facial system is
 \[
   (\calL_{f,u})_w
    \   =\
  \begin{bmatrix}
     c_{01}x_2+c_{11}x_1x_2+c_{21}x_1^2x_2\\
     c_{11}x_2+2c_{21}x_1x_2\\
     c_{01}+c_{11}x_1+c_{21}x_1^2-\lambda x_2
  \end{bmatrix}
    \  =\
  \begin{bmatrix} f_w \\
    \partial_1 f_w  \\
    \partial_2 f_w -\lambda x_2
  \end{bmatrix}\ .
\]
Observe that  $h_w(\mathcal A)=-1$ and that we have
 \begin{align}
    h_w(\mathcal A)\cdot f_w\  =\  -f_w{\ }&=\ w_1\cdot x_1\cdot \partial_1f_w + w_2\cdot  x_2\cdot \partial_2f_w) \nonumber\\
     &=\  0\cdot x_1\cdot (\partial_1f_w) + (-1)\cdot  x_2\cdot (\partial_2f_w)\ =\ x_2\partial_2 f\,.\label{Eq:Euler_Ex}
\end{align}
 This is an instance of \demph{Euler's formula for quasihomogeneous polynomials} (Lemma \ref{lemma_derivative}).
 If $(\lambda,x)$ is a solution to $(\calL_{f,u})_w=0$, then  the third equation becomes $\partial_2 f = \lambda x_2$.
 Substituting this into~\eqref{Eq:Euler_Ex} gives $0 = -f_w = \lambda x_2^2$, which has no solutions in $(\mathbb C^\times)^3$.
 One of the cases in the proof of Theorem \ref{Thm:EmptyFaces} exploits Euler's formula in a similar way.  \hfill$\diamond$

\section{The facial systems of the Lagrange multiplier equations are empty}\label{S:EDD=MV}
Before giving a proof of Theorem~\ref{Thm:EmptyFaces}, we present two lemmas to help understand the support of $f$ and its interaction with
derivatives of $f$, and then make some observations about the facial system $(\calL_{f,u})_w$.

Let $f\in\CC[x_1,\dotsc,x_m]$ be a polynomial with support  $\defcolor{\calA}\subset \NN^{n}$, which is the set of the exponents of monomials
of $f$.
We assume that $0\in\calA$.
As before we write ${\partial_i\calA}\subset \NN^{n}$ for the support of the partial derivative
$\partial_i f$.
For $w\in\ZZ^n$, the linear function $x\mapsto w\cdot x$ takes minimum values on $\calA$ and on $\partial_i\calA$, which we denote by
 \begin{equation}\label{def_optimal_value}
  \defcolor{h^*}\ =\ h_w(\calA)\ :=\ \min_{a\in \calA} w\cdot a \qquad\text{and}\qquad
  \defcolor{h^*_i}\ =\ h_w(\partial_i\calA)\ :=\ \min_{a\in \partial_i \calA} w\cdot a\,.
 \end{equation}
(We suppress the dependence on $w$.)
Since $0\in\calA$, we have $h^*\leq 0$.
Also, if $h^*=0$ and if there is some $a\in\calA$ with $a_i>0$, then $w_i\geq 0$.
%
%

Recall that the subsets of $\calA$ and $\partial_i\calA$ where the linear function $x\mapsto w\cdot x$
is minimized are their \demph{faces} exposed by $w$,
 \begin{equation}\label{def_optimal_sets}
  \calA_w\ :=\ \{a\in \calA \mid w\cdot a = h^*\} \quad\text{and}\quad
  {(\partial_i\calA)_w}\ :=\ \{a\in \partial_i\calA \mid w\cdot a = h^*_i\}\,.
 \end{equation}
 {The proof below of Lemma~\ref{lemma_helpful} shows that $\partial_i(\calA_w)\subset (\partial_i\calA)_w$ with equality when
   $\emptyset\neq \partial_i(\calA_w)$.}
 As in~\eqref{def_g_w} we denote by $f_w$ the restriction of $f$ to $\calA_w$, and similarly $(\partial_i f)_w$ denotes
the restriction of the partial derivative $\partial_i f$ to $\calA_w$.
{The $i$th partial derivative of $f_w$ is $\partial_i(f_w)$.}

Our proof of Theorem~\ref{Thm:EmptyFaces} uses the following two results.

\begin{lemma}\label{lemma_helpful}
  For each $1\leq i\leq n$, {we have $h^*_i\geq h^*-w_i$.}
  If $\partial_i f_w \neq 0$, then $\partial_i(f_w) = (\partial_i f)_w$ and $h_i^* = h^* - w_i$.
\end{lemma}
{In the following, we write  \defcolor{$\partial_i f_w$} for $\partial_i(f_w)$ and write $\defcolor{\partial_i\calA_w}$
   for $(\partial_i\calA)_w$.}
\begin{proof}
  Fix $1\leq i\leq n$.
  {Let $a\in\partial_i\calA$.
  Then $a+\be_i\in\calA$ and so $h^*\leq w\cdot (a+\be_i)=w\cdot a +w_i$.
  Thus $w\cdot a\geq h^*-w_i$.
  Taking the minimum over $a\in\partial_i\calA$ gives $h^*_i \geq h^*-w_i$.}

 {Suppose now that $\emptyset\neq \partial_i(\calA_w)$.
   Let $a\in\partial_i(\calA_w)$.
   Then, we have $a+\be_i\in\calA_w$ and $h^*=w\cdot(a+\be_i)=w\cdot a + w_i$.
   But then $h^*-w_i=w\cdot a\geq h^*_i$, which implies that $h_i^* = h^* - w_i$.
   It also implies that $w\cdot a=h^*_i$.
   Since $\calA_w\subset\calA$, we have that $a\in\partial_i\calA$.
   As~$w\cdot a=h^*_i$, we conclude that~$a\in (\partial_i\calA)_w$.
   This proves the inclusion $\partial_i(\calA_w)\subset(\partial_i\calA)_w$.}

 {For the other inclusion, suppose that $\partial_i(\calA_w)\neq\emptyset$.
   As we showed, it holds that $h_i^* = h^* - w_i$.
   Let $a\in (\partial_i\calA)_w$.
   Then $w\cdot a= h^*_i$ and as $a\in\partial_i\calA$, we have $a+\be_i\in\calA$.
   But then, we have $w\cdot(a+\be_i)=h^*_i+w_i=h^*$, so that $a+\be_i\in\calA_w$.
   We conclude that $a\in\partial_i(\calA_w)$.}

 {To complete the proof, observe that $\partial_i f_w \neq 0$ is equivalent to $\partial_i(\calA_w)\neq\emptyset$,
   and that $\partial_i(f_w)$ and $(\partial_i f)_w$ are subsums of $\partial_i f$
   over terms corresponding to $\partial_i(\calA_w)$ and to $(\partial_i\calA)_w$, respectively.}
\end{proof}

The restriction $f_w$ of $f$ to the face of $\calA$ exposed by $w$ is quasihomogeneous with respect to the weight $w$,
and thus it satisfies a weighted version of Euler's formula.

\begin{lemma}[Euler's formula for quasihomogeneous polynomials]
  \label{lemma_derivative}
  For $w\in\ZZ^n$ we have
  \[
  h^*\cdot f_w\ =\ \sum_{i=1}^n w_i  x_i \partial_i f_w\,.
  \]
\end{lemma}
\begin{proof}
  For a monomial $x^a$ with $a\in\ZZ^n$ and $1\leq i\leq n$,  we have that $x_i\partial_i x^a= a_i x^a$.
  Thus
  \[
    \sum_{i=1}^n w_i x_i \partial_i x^a\ =\
    \sum_{i=1}^n w_i a_i \  x^a\ =\ (w\cdot a) x^a\,.
  \]
  The statement follows because for $a\in\calA_w$ (the support of $f_w$), $w\cdot a= h^*$.
\end{proof}

Our proof of Theorem \ref{Thm:EmptyFaces} investigates facial systems $(\calL_{f,u})_w$ for $0\neq w\in\ZZ^{n+1}$ with the aim of showing
that for $f$ general given its support $\calA$, no facial system has a solution.
Recall from~\eqref{Lagrange_multiplier_equation} that the Lagrange multiplier equations for the Euclidean distance problem are
\[
  \calL_{f,u}(\lambda,x_1,\ldots,x_n)\ =\
    \begin{bmatrix}
       f(x_1,\ldots,x_n)\\
       \partial_1 f - \lambda(u_1-x_1)\\
       \vdots\\
       \partial_n f - \lambda(u_n-x_n)
    \end{bmatrix}\  =\ 0\,.
\]
Fix $0\neq w=(v,w_1,\ldots,w_{n})\in\ZZ^{n+1}$.
The {initial} coordinate of $w$ is $v\in\ZZ$. It has index $0$ and corresponds to the variable $\lambda$.

The first entry of the facial system $(\calL_{f,u})_{w}$ is $f_{w}$.
The shape of the remaining entries depends on $w$ as follows.
Recall from~\eqref{def_optimal_value} that
{we have set $h^* := \min\{w\cdot a\mid a\in \calA\}$ and $h^*_i := \min\{w\cdot a\mid a\in \partial_i \calA\}$.}
%
%
As $v$ and $v+w_i$ are the weights of the monomials $\lambda u_i$ and $\lambda x_i$, respectively, there are seven possibilities for
each of these remaining entries,
 \begin{equation}\label{shapes}
  \left(\partial_i f - \lambda(u_i-x_i)\right)_{w}\  =\
  \begin{cases}
   (\partial_i f)_{w}&
   \text{ if } h^*_i < \min\{v, v + w_i\}\,,
     \\[0.1em]
     (\partial_i f)_{w} - \lambda (u_i-x_i)& \text{ if } h^*_i =  v \text{ and } w_i = 0\,,
     \\[0.1em]
     (\partial_i f)_{w} - \lambda u_i& \text{ if } h^*_i = v  \text{ and } w_i > 0\,,
     \\[0.1em]
      (\partial_i f)_{w} + \lambda x_i& \text{ if } h^*_i = v + w_i \text{ and } w_i < 0\,,
     \\[0.1em]
    - \lambda (u_i-x_i)& \text{ if }h^*_i    > v \text{ and } w_i = 0\, ,
     \\[0.1em]
    - \lambda u_i& \text{ if }h^*_i  > v \text{ and } w_i > 0\, ,
     \\[0.1em]
     \lambda x_i& \text{ if }h^*_i > v + w_i \text{ and } w_i < 0\, .
 \end{cases}
\end{equation}
Note that if one of the polynomials $f_w$ or $\left(\partial_i f - \lambda(u_i-x_i)\right)_{w}$ is a
monomial, then $(\calL_{f,u})_w$ has no solutions in $(\CC^\times)^{n+1}$.

For a subset $\calI\subset\{1,\dotsc,n\}$ and a vector $u\in\CC^n$, let
$\defcolor{u_\calI}:=\{ u_i\mid i\in\calI\}$ be the components of $u$ indexed by $i\in\calI$.
We similarly write \defcolor{$w_\calI$} for $w\in\ZZ^n$ and \defcolor{$x_\calI$} for variables $x\in\CC^n$, and
write \defcolor{$\CC^\calI$} for the corresponding subspace of $\CC^n$.

We recall Theorem~\ref{Thm:EmptyFaces}, before we give a proof.
\medskip

\noindent{\bf Theorem~\ref{Thm:EmptyFaces}.}
  {\it
    Suppose that $f$ is general given its support $\calA$, that $0\in\calA$, and that $u\in\RR^n$ is general.
    For any nonzero $w\in\ZZ^{n+1}$,  the facial system $(\calL_{f,u})_w$ has no solutions in~$(\CC^\times)^{n+1}$.}\medskip

\begin{proof}
  Let $0\neq w=(v,w_1,\ldots,w_n)\in\ZZ^{n+1}$. As before, $v$ corresponds to the variable $\lambda$ and~$w_i$ to $x_i$.
  We argue by cases that depend upon $w$ and $\calA$, showing that in each case, for a general polynomial $f$ with support $\calA$, the
  facial system has no solutions in~$(\CC^\times)^{n+1}$.
  Note that the last two possibilities in~\eqref{shapes} do not occur as they give monomials.
  As $f$ {has support $\calA$}, if $\partial_if_w=0$, then $\calA_w\subset\{a\in\NN^n\mid a_i=0\}$.

  We distinguish three cases.\smallskip

\noindent{\bf Case 1 (the constant case):}
Suppose that $\partial_i f_{w} = 0$ for all $1\leq i\leq n$.
Then $f_w$ is the constant term of $f$.
Since $0\in\calA$, this is nonvanishing for $f$ general and the facial system~$(\calL_{f,u})_w$ has no solutions. \medskip

For the next two cases we may assume that there is a partition
$\defcolor{\calI} \sqcup\defcolor{\calJ} =\{1,\ldots,n\}$ with $\calI$ nonempty such that
$\partial_i f_{w} \neq 0$ for $i\in \calI$ and $\partial_j f_{w} = 0$ for $j\in\calJ$.
By Lemma \ref{lemma_helpful} we have
 \begin{equation}\label{h_for_I}
   h_i^*\ =\ h^* - w_i  \  \text{for all}\  i\in\calI\,.
 \end{equation}
 As $j\in\calJ$ implies that $\partial_j f_{w} = 0$, we see that if $a\in\calA_w$, then $a_\calJ=0$.
 This implies that {$f_w$ is a polynomial in only the variables $x_\calI$, that is, $f_w\in\CC[x_\calI]$.}\smallskip

\noindent{\bf Case 2 (triangular systems):}
Suppose that for $i\in\calI$, $w_i\geq 0$, that is, $w_{\calI}\geq 0$.
We claim that this implies $w_{\calI} = 0$.
To see this, let $a\in\calA_w$.
{As we observed,} $a_\calJ=0$.
We have
\[
  0\ \geq h^*\ =\ w\cdot a\ =\ w_{\calI} \cdot a_{\calI}\ \geq\ 0\,.
\]
Thus $h^*  = w_{\calI} \cdot a_{\calI} = 0$, which implies that $0\in\calA_w$.
Let $i\in\calI$.
Since $\partial_i f_{w} \neq 0$, there exists some $a\in \calA_w$ with $a_i>0$.
Since $w_{\calI} \cdot a_{\calI} = 0$ for all $a\in\calA_w$, we conclude that $w_i=0$.

Let $i\in\calI$.
By Lemma \ref{lemma_helpful}, we have $h_i^* = h^* - w_i$, so that $h^*_i =  h^*=0$,
and we also have~$(\partial_i f)_{w} = \partial_i f_{w}$.
As $w_i=0$, the possibilities from~\eqref{shapes} become
 \begin{equation*}
  \left(\partial_i f - \lambda(u_i-x_i)\right)_{w}\ =\
   \begin{cases}
     \partial_i f_{w}&
     \text{ if } v > 0\,, \\[0.1em]
     \partial_i f_{w} - \lambda (u_i-x_i)& \text{ if } v = 0\,,
     \\[0.1em]
     - \lambda (u_i-x_i)& \text{ if }  v< 0\,.
   \end{cases}
 \end{equation*}

We consider three subcases of $v<0$, $v>0$, and $v=0$ in turn.
Suppose first that $v<0$ and
that $(\lambda,x)\in(\CC^\times)^{n+1}$ is a solution to $(\calL_{f,u})_w$.
As $\lambda\neq 0$ and we have $\lambda(u_i-x_i)=0$ for all $i\in\calI$, we conclude that $x_{\calI} = u_{\calI}$.
Since $f_w\in\CC[x_\calI]$ is a general polynomial {with} support $\calA_w$ and $u$ is general, we do not have $f_w(u_{\calI})=0$.
Thus $(\calL_{f,u})_{w}$ has no solutions when $v<0$.\smallskip

Suppose next that $v>0$.
Then the subsystem of $(\calL_{f,u})_w=0$ involving $f_w$ and the equations indexed by $\calI$ is
 \begin{equation}\label{Eq:singular_face}
   f_w\ =\ \partial_i f_w\ =\ 0\,,\qquad\mbox{for }i\in\calI\,.
 \end{equation}
 As $f_w\in\CC[x_\calI]$, the system of equations~\eqref{Eq:singular_face} implies that the hypersurface
 $\calV_{(\CC^\times)^\calI}(f_w)\subset(\CC^\times)^\calI$ is singular.
However, since $f_w$ is general, this hypersurface must be smooth.
Thus $(\calL_{f,u})_w$ has no solutions when $v>0$.\smallskip

The third subcase of $v=0$ is more involved.
When $v=0$, the subsystem of $(\calL_{f,u})_w$ consisting of $f_w$ and the equations indexed by $\calI$ is
 \begin{equation}\label{Eq:Isubsystem}
   f_w\ =\ \partial_i f_w-\lambda(u_i-x_i)\ =\ 0
   \qquad\mbox{for }i\in\calI\,.
 \end{equation}
 As $f_w\in\CC[x_\calI]$ and $0\in\calA_w$, this is the system $(\calL_{f,u})_w$ in $\CC_\lambda\times\CC^\calI$ for the critical points of
 Euclidean distance from $u_\calI\in\CC^\calI$ to the hypersurface $\calV_{\CC^\calI}(f_w)\subset\CC^\calI$.
 Thus $(\calL_{f,u})_w$ is triangular;
 to solve it, we first solve~\eqref{Eq:Isubsystem}, and then consider the equations in
 $(\calL_{f,u})_w$ indexed by $\calJ$.

 Since $\partial_j f_w=0$ for $j\in\calJ$, the remaining equations are independent of $u_\calI$ and $f_w$.
 We will see that they are also triangular.

 {Since $h^*=0$, if $a\in\calA\smallsetminus\calA_w$, then $w\cdot a>0$.}
 Let $j\in\calJ$.
 We earlier observed that if $a\in\calA_w$ then $a_j=0$ and we defined $h^*_j$ to be the minimum $\min\{ w\cdot a\mid a\in\partial_j\calA\}$.
 Since if $a\in\partial_j\calA$, then $a+\be_j\in\calA$, we have that $a+\be_j\in\calA\smallsetminus\calA_w${.}
 {We arrive at} $w\cdot(a+\be_j)>0${, which implies} that~$w\cdot a>-w_j$.
 Taking the minimum over $a\in\partial_j\calA$ implies that $h^*_j>-w_j$.

 Consider now the members of the facial system $(\calL_{f,u})_w$ indexed by $j\in\calJ$.
 Since $v=0$ and~$h^*_j>-w_j$, the second and fourth possibilities for $(\partial_jf-\lambda(u_j-x_j))_w$ in~\eqref{shapes} do not occur.
 Recall that the last two possibilities also do not occur.
 As $v=0$, we have three cases
 \begin{equation}\label{Eq:Newshapes}
  \left(\partial_j f - \lambda(u_j-x_j)\right)_w\  = \
  \begin{cases}
   (\partial_j f)_w&
   \text{ if } h^*_j < \min\{0, w_j\}\,,
     \\[0.1em]
     (\partial_j f)_w - \lambda u_j& \text{ if } h^*_j = 0  \text{ and } w_j > 0\,,
     \\[0.1em]
    - \lambda (u_j-x_j)& \text{ if }h^*_j    > 0 \text{ and } w_j = 0\, .
 \end{cases}
\end{equation}
If the first case holds for some $j\in\calJ$, then as $h^*_j>-w_j$, we have $w_j>0$.
Since $w_j\geq 0$ in the other cases, we have $w_j\geq 0$ for all $j\in \calJ$.
As we showed earlier that $w_\calI=0$, we have  $w\geq 0$.
But then as $\partial_j\calA\subset\NN^n$, we have $h^*_j\geq 0$ for all $j\in\calJ$.
In particular, the first case in~\eqref{Eq:Newshapes}---in which $h^*_j<0$---does not occur.
Thus the only possibilities for the~$j$th component of $(\calL_{f,u})_w$ are the second or the third cases in~\eqref{Eq:Newshapes}, so that
$w_\calJ\geq 0$.

Let us further partition $\calJ$ according to the vanishing of $w_j$,
\[
  \calK\ :=\ \{k\in\calJ\mid w_k=0\}
   \qquad\mbox{and}\qquad
  \calM\ :=\ \{m\in\calJ\mid w_m>0\}\,.
\]
Every component of $w_\calM$ is positive and $w_\calI=w_\calK=0$. Moreover, the second entry in \eqref{Eq:Newshapes} shows that $h^*_m=0$ for all $m\in\calM$. We conclude from this
that no variable in~$x_\calM$ occurs in~$(\partial_m f)_w$, for any $m\in\calM$.

Let us now consider solving $(\calL_{f,u})_w$, using triangularity.
Let \defcolor{$(\lambda,x_\calI)$} be a solution to the subsystem~\eqref{Eq:Isubsystem} for critical points of the
Euclidean distance from $u_\calI$ to $\calV_{\CC^\calI}(f_w)$ in $\CC^\calI$.
We may assume that $\lambda\neq 0$ as $f_w$ is general.
Then the subsystem corresponding to $\calK$ gives $x_k=u_k$ for $k\in\calK$.
Let $m\in\calM$.
Since $(\partial_m f)_w$ only involves $x_\calI$ and $x_\calK$, substituting these values into $(\partial_m f)_w$
gives a constant, which cannot be equal to $\lambda u_m$ for general $u_m\in\CC$.
As $w\neq 0$, we cannot have $\calM=\emptyset$, so this last case occurs.
Thus $(\calL_{f,u})_w$ has no solutions when $v=0$.\medskip

\noindent{\bf Case 3 (using the Euler formula):}
Let us now consider the case where there is some index $i\in\calI$ with $w_i<0$ and suppose that the facial system $(\calL_{f,u})_w$
has a solution.
Let $i\in\calI$ be {such} an index with $w_i<0$.
As the facial system has a solution, the last possibility in~\eqref{shapes} for~$(\partial_i f-\lambda(u_i-x_i))_w$ does not occur.
Thus either first or the fourth possibility occurs.
Hence, $h_i^*\leq w_i+v<v$, as $w_i<0$.
By~\eqref{h_for_I}, we have $h^*=h^*_i+w_i\leq 2w_i+v < v$.

For any $i\in\calI$, we have $h^*_i=h^*-w_i<v-w_i$, by~\eqref{h_for_I}.
Thus if $w_i\geq 0$, then  $h^*_i<v$.
As we obtained the same inequality when $w_i<0$, we conclude that for all $i\in\calI$ we have $h_i^*<v$.
Thus only the first or the fourth possibility in~\eqref{shapes} occurs for $i\in\calI$.
That is,
 \begin{equation}\label{proof_eq1}
   (\partial_i f - \lambda(u_i-x_i))_{w}
    =
    \begin{cases}
      \partial_i f_{w} &
       \text{ if } h^* - w_i < \min\{v, w_i + v\} \,,
   \\[0.1em]
   \partial_i f_{w} + \lambda x_i& \text{ if } h^* =  2w_i + v \mbox{ and } w_i<0\,.
  \end{cases}
 \end{equation}
These cases further partition $\calI$  into sets $\calK$ and $\calM$, where
\begin{align*}
   \defcolor{\calK}{\ } &:=\
    \{k\in\calI\mid  h^* - w_k < \min\{v, w_k + v\} \}\,\; \mbox{and}\\
   \defcolor{\calM}\ &:=\
   \{m\in\calI\mid  h^* =  2w_m + v \mbox{ and } w_m<0\}\, .
\end{align*}
For $k\in\calK$ the corresponding equation in $(\calL_{f,u})_w=0$  is $\partial_k f_w=0$ and for $m\in\calM$ it is~$\partial_m f_{w} + \lambda x_m=0$.
If $\calM=\emptyset$, then $\calK=\calI$ and the subsystem of $(\calL_{f,u})_w$ consisting of $f_w$ and the equations indexed by
$\calI$ is~\eqref{Eq:singular_face}, which has no solutions as we already observed.

Now suppose that $\calM\neq \emptyset$.
Define $\defcolor{w^*}:=\min\{w_i\mid i\in\calI\}$.
Then $w^*<0$.
Moreover, by~\eqref{proof_eq1} we have that if $m\in\calM$, then $w_m=\frac{1}{2}(h^*-v)$.
Thus, $w_m=w^*$ for every $m\in\calM$.
Suppose that $(\lambda,x)$ is a solution to $(\calL_{f,u})_w$.
For $k\in\calK$, we have $\partial_k f_w(x)=0$ and for $m\in\calM$, we have that~$\partial_m f_{w}(x)=-\lambda x_m$.
Then by Lemma~\ref{lemma_derivative}, we get
 \[
   0\ =\ h^* f_w(x)\ =\ \sum_{i\in\calI} w_i \, x_i \, \partial_i f_w(x)\ =\ -\lambda w^*\sum_{m\in\calM} x_m^2 \ .
 \]
 The last equality uses that $\calI=\calK\sqcup\calM$.
 Since $\lambda\neq 0$ and $w^*\neq 0$, we have $\sum_{m\in\calM} x_m^2= 0$.
 Let \defcolor{$Q$} be this quadratic form.
 Then the point $x_\calI$ lies on both hypersurfaces $\calV(f_w)$ and $\calV(Q)$.
 Since $\partial_k f_w(x_\calI)=\partial_k Q=0$ for $k\in\calK$ and
 $2\partial_m f_w(x_\calI)=\lambda\partial_m Q$ for $m\in\calM$, we see that the two hypersurfaces meet non-transversely at $x_\calI$.
 But this contradicts $f_w$ being general.
 Thus there are no solutions to $(\calL_{f,u})_{w}=0$ in this last case.

 This completes the proof of Theorem~\ref{Thm:EmptyFaces}.
\end{proof}

\section{The Euclidean Distance Degree of a rectangular parallelepiped}\label{sec:rectangular_parallelepiped}
Let $a=(a_1,\dotsc,a_n)$ be a vector of nonnegative integers and recall from~\eqref{def_box}  the definition of the rectangular
parallelepiped:
\[
B(a)\ =\ [0,a_1]\times \dotsb\times [0,a_n]\ \subset\  \RR^n\,.
\]
We consider the {Euclidean Distance Degree} of a general polynomial whose Newton polytope is $B(a)$, with the goal of
proving Theorem~\ref{Th:EDD_of_box}.
{We consider polytopes in $\RR^n$, such as $B(a)$, as polytopes in $\RR^{n+1}$, using the identification of $\RR^n$ with
  $\{0\}\times\RR^n\subset\RR^{n+1}$.}

Recall that ${\be_i}:=(0,\dotsc,1,\dotsc,0)$ is the $i$th standard unit vector in $\RR^n$
(the unique 1 is in the $i$th position).
The $0$-th unit vector $\be_0$ corresponds to the variable $\lambda$.
Let $f$ be a general polynomial with Newton polytope $B(a)$.
Then the Newton polytope of the partial derivative~$\partial_i f$ is $B(a_1,\dotsc,a_i{-}1,\dotsc,a_n)$.

For each $1\leq i\leq n$, let $\defcolor{P_i(a)}{\subset\RR^{n+1}}$ be the convex hull of $B(a_1,\dotsc,a_i{-}1,\dotsc,a_n)$
and the two points $\be_0$ and $\be_0+\be_i$.
Then $P_i(a)$ is the Newton polytope of $\partial_i f - \lambda(u_i-x_i)$.
Consequently, $B(a),P_1(a),\dotsc,P_n(a)$ are the Newton polytopes of the Lagrange multiplier equations~\eqref{Lagrange_multiplier_equation}.

Recall that for each $1\leq k\leq n$, $e_k(a)$ is the elementary symmetric polynomial of degree $k$  evaluated at $a$.
It is the sum of all square-free monomials in $a_1,\dotsc,a_n$.
Let us write
 \begin{equation}\label{Eq:Ea}
   \defcolor{E(a)}\ := \sum_{k=1}^n  k! \,e_k(a)\,.
 \end{equation}
The main result in this section is the following mixed volume computation.
It and Theorem~\ref{Th:EDD=MV} together imply Theorem~\ref{Th:EDD_of_box}.

\begin{theorem}\label{Th:MV_for_box}
  With these definitions, $\MV(B(a),P_1(a),\dotsc,P_n(a)) = E(a)$.
\end{theorem}

Our proof of Theorem~\ref{Th:MV_for_box} occupies Section~\ref{SS:Proof}, and it depends upon lemmas and definitions collected in
Sections~\ref{SS:Defs} and \ref{sec:pyramids}.
One technical lemma from Section~\ref{sec:pyramids} is proven in~Section~\ref{SS:Pyramids}.

\subsection{Mixed volumes}\label{SS:Defs}

Let $m$ be a positive integer.
The \demph{Minkowski sum} of two polytopes $P,Q$ in $\RR^m$ is the sum of all pairs of points, one from each of $P$ and $Q$,
\[
   \defcolor{P+Q}\ :=\ \{p+q\mid p\in P\ \mbox{ and }\ q\in Q\}\,.
\]
Let $m$ be a positive integer.
As explained in~\cite[Sect.~IV.3]{Ewald},
mixed volume is a nonnegative function $\MV(Q_1,\dotsc,Q_m)$ of polytopes $Q_1,\dotsc,Q_m$ in $\RR^m$ that is characterized by three
properties: \smallskip

{\bf Normalization.} If $Q_1=\dotsb=Q_m=Q$, and $\Vol(Q)$ is the Euclidean volume of $Q$, then
\[
\MV(Q_1,\dotsc,Q_m)\ =\ m!\Vol(Q)\,.
\]

{\bf Symmetry.} If $\sigma$ is a permutation of $\{1,\dotsc,m\}$, then
\[
\MV(Q_1,\dotsc,Q_m)\ =\ \MV(Q_{\sigma(1)},\dotsc,Q_{\sigma(m)})\,.
\]

{\bf Multiadditivity.} If $Q'_1$ is another polytope in $\RR^m$, then
\[
    \MV(Q_1+Q'_1,Q_2,\dotsc,Q_m)\ =\ \MV(Q_1,Q_2,\dotsc,Q_m)\ +\ \MV(Q'_1,Q_2,\dotsc,Q_m)\,.
\]

Mixed volume decomposes as a product when the polytopes possess a certain triangularity
(see~\cite[Lem.~6]{ThSt} or~\cite[Thm.~1.10]{Esterov}).
We use a special case.
For a positive integer~$b$, write \defcolor{$[0,b\,\be_i]$} for the interval of length $b$ along the $i$th axis in~$\RR^{m}$.
For each $1\leq j\leq m$, let $\defcolor{\pi_j}\colon\RR^{m}\to\RR^{m-1}$ be the projection along the coordinate direction~$j$.

\begin{lemma}\label{L:intervals}
  Let $Q_1,\dotsc,Q_{m-1}\subset\RR^m$ be polytopes, $b$ be a positive integer, and $1\leq j\leq m$.
  Then
  \[
  \MV(Q_1,\dotsc,Q_{m-1},[0,b\,\be_j])\ =\ b\,\MV(\pi_j(Q_1),\dotsc,\pi_j(Q_{m-1}))\,.
  \]
\end{lemma}
\begin{proof}
 We paraphrase the proof in~\cite{Esterov}, which is bijective and algebraic.
 Consider a system $g_1,\dotsc,g_m$ of general polynomials with Newton polytopes $Q_1,\dotsc,Q_{m-1},[0,b\,\be_j]$, respectively.
 As $g_m$ is a univariate polynomial of degree $b$ in $x_j$, $g_m(x_j)=0$ has $b$ solutions.
 For each solution $x^*_j$, if we substitute $x_j=x^*_j$ in $g_1,\dotsc,g_{m-1}$, then we obtain general polynomials with Newton polytopes
 $\pi_j(Q_1),\dotsc,\pi_j(Q_{m-1})$.
 Thus there are $\MV(\pi_j(Q_1),\dotsc,\pi_j(Q_{m-1}))$ solutions to our original system for each of the $b$ solutions to $g_m(x_j)=0$.
\end{proof}

%
\subsection{Pyramids}\label{sec:pyramids}
Let $1\leq m\leq n$ and $a=(a_1,\dotsc,a_m)$ be a vector of positive integers.
The \demph{small rectangular parallelepiped} is $\defcolor{B(a)}:=[0,a_1]\times\cdots\times [0,a_m]$.
It is the Minkowski sum of intervals:
\[
B(a)\ =\ [0,a_1\be_1]+\dotsb+[0,a_k\be_m]\,.
\]
Its Euclidean volume is $a_1\dotsb a_m$, the product of its side lengths.
{This is embedded in $\RR^{m+1}$ as $\{0\}\times B(a)$.}

As before, $\defcolor{P_i(a)}$ is the convex hull of $B(a_1,\dotsc,a_i{-}1,\dotsc,a_m)$ and~$\be_0+[0,\be_i]$.
Define $\defcolor{\Pyr(a)}$ to be the pyramid with base the rectangular parallelepiped $B(a)$ and apex~$\be_0$, this is the convex hull of
$B(a)$ and $\be_0$.
For each $j=1,\ldots,m$ we have the projection $\pi_j:\mathbb R^{m}\to \mathbb R^{m-1}$ along the $j$th coordinate,
so that $\pi_j(a)=(a_1,\dotsc,a_{j-1}\,,\,a_{j+1},\dotsc,a_m)$.
We then have that $\pi_j(B(a))=B(\pi_j(a))$.
The following is immediate from the definitions.

\begin{lemma}\label{L:Projections}
 Let $a=(a_1,\dotsc,a_m)$ and $1\leq i,j\leq m$.
  Then we have
  \[
    \pi_j(P_i(a)) \mbox{\ }=\ \left\{\begin{array}{rcl} P_i(\pi_j(a)) &\ &\mbox{if }\; i\neq j\\
                                                     \Pyr(\pi_j(a)) &\ &\mbox{if }\; i = j \rule{0pt}{14pt}
                                    \end{array}\right. .
   \]
\end{lemma}

We now have the following lemma.
{Recall the definition~\eqref{Eq:Ea} of $E(a)$.}

\begin{lemma}\label{L:MV_Pyramids}
  We have $\MV(\Pyr(a),P_1(a),\dotsc,P_m(a))\ =\ 1 + E(a)$.
\end{lemma}

{We prove this} in Section~\ref{SS:Pyramids}.

\subsection{Proof of Theorem~\ref{Th:MV_for_box}}\label{SS:Proof}

Since $B(a)$ is the Minkowski sum of the intervals $[0,a_i\be_i]$ for $1\leq i\leq n$, multiadditivity and Lemma~\ref{L:intervals}
give
\begin{align}
   \MV(B(a),P_1(a),\dotsc,P_n(a))\mbox{\ }&=\                 \label{Eq:MixedVol}    %
   \sum_{j=1}^n \MV([0,a_j\be_j], P_1(a),\dotsc,P_n(a))\\
   &=\  \sum_{j=1}^n a_j \MV( \pi_j(P_1(a)),\dotsc,\pi_j(P_n(a)))\,.   \nonumber%
\end{align}
By Lemma~\ref{L:Projections}, the $j$th term is
\[
 a_j \MV( P_1(\pi_j(a)),\dotsc,P_{j-1}(\pi_j(a))\,,\, \Pyr(\pi_j(a))\,,\, P_{j+1}(\pi_j(a)),\dotsc,P_n(\pi_j(a)))\,.
\]
Applying symmetry and Lemma~\ref{L:MV_Pyramids} with $m=n{-}1$, this is $a_j(1+E(\pi_j(a)))${, where $E(\cdot)$ is defined in~\eqref{Eq:Ea}}.
Thus the mixed volume~\eqref{Eq:MixedVol}  is
\[
  e_1(a)\ +\ \sum_{k=1}^{n-1} k!\ \sum_{j=1}^n a_j e_k(\pi_j(a))
  \ = E(a)\,.
\]
The equality {in this formula} follows from the identity,
 \[
  \sum_{j=1}^n a_j e_k(\pi_j(a))\ =\
  \sum_{j=1}^n a_j e_k(a_1,\dotsc,a_{j-1}\,,\, a_{j+1},\dotsc, a_n)\ =\
   (k{+}1) e_{k+1}(a)\,.
\]
This finishes the proof of Theorem~\ref{Th:MV_for_box}.\qed
%
%

\subsection{Proof of Lemma~\ref{L:MV_Pyramids}}\label{SS:Pyramids}

We use Bernstein's Theorem {to show} that a general polynomial system with support
$\Pyr(a),P_1(a),\dotsc,P_m(a)$ has $1+E(a)$ solutions in the torus $(\CC^\times)^{m+1}$, where $a=(a_1,\dotsc,a_m)$ is a vector of positive integers.

A general polynomial with Newton polytope $\Pyr(a)$ has the form $c\lambda+f$, where $f$ has Newton polytope $B(a)$ and $c\neq 0$.
Here, $\lambda$ is a variable with exponent $\be_0$.
Dividing by~$c$, we may assume that the polynomial is monic in $\lambda$.
Similarly, as $P_i(a)$ is the convex hull of $B(a_1,\dotsc,a_i{-}1,\dotsc,a_m)$ and~$\be_0+[0,\be_i]$, a general polynomial with support
$P_i(a)$ may be assumed to have the form $\lambda\ell_i(x_i) + f_i(x)$, where $f_i$ has Newton polytope $B(a_1,\dotsc,a_i{-}1,\dotsc,a_m)$ and
$\defcolor{\ell_i(x_i)}:=c_i+x_i$ is a linear polynomial in $x_i$ with $c_i\neq 0$.

We may therefore assume that a general system of polynomials with the given support has the form
  \begin{equation}\label{Eq:pyramidPolys}
    \lambda-f\,,\ \,
    \lambda\ell_1(x_1)+f_1\,,\ \dotsc\,,\  \lambda\ell_m(x_m)+f_m\,,
  \end{equation}
where $f$ is a general polynomial with Newton polytope $B(a)$ and for each $1\leq i\leq m$, $f_i$ is a general polynomial with Newton
polytope  $B(a_1,\dotsc,a_i{-}1,\dotsc,a_m)$.
We show that $1+ E(a)$ is the number of common zeros in $(\CC^\times)^{n+1}$ of the polynomials in~\eqref{Eq:pyramidPolys}.

Using the first polynomial to eliminate $\lambda$ from the rest shows that solving the
system~\eqref{Eq:pyramidPolys} is equivalent to solving the system
  \begin{equation}\label{Eq:smallerSystem}
    F\ \colon\  f_1+\ell_1(x_1)f\,,\ \dotsc\,,\ f_m + \ell_m(x_m) f\,,
  \end{equation}
 which is in the variables $x_1,\dotsc,x_m$, as $z\mapsto (f(z),z)$ is a bijection between the solutions~$z$ to~\eqref{Eq:smallerSystem}
 and the solutions to~\eqref{Eq:pyramidPolys}.
 We show that the number of common zeroes to~\eqref{Eq:smallerSystem} is~$1+E(a)$, when $f,f_1,\dotsc,f_m$ are
 general given their Newton polytopes.

 Unlike the system~\eqref{Eq:pyramidPolys}, the system $F$ is not general given its support.
 Nevertheless, we will show that no facial system has any solutions.
 Then, by Bernstein's Other Theorem, its number of solutions is the corresponding mixed
 volume, which we now compute.

 Since $B(a_1,\dotsc,a_i{-}1,\dotsc,a_m)\subset B(a)$, the Newton polytope of $f_i+\ell_i(x_i)f$ is $B(a)+[0,\be_i]$.
 Thus the mixed volume we seek is
 \[
   \MV(B(a)+[0,\be_1],\dotsc,B(a)+[0,\be_m])\ =\
   \sum_{\calI\subset\{1,\dotsc,m\}} |\calI|! \prod_{i\in \calI} a_i\ =\ 1 + E(a)\,.
 \]
 To see this, first observe that the second equality is the definition of $E(a)$.
 For the first equality, consider expanding the mixed volume using multilinearity.
 This will have summands indexed by  subsets $\calI$ of $\{1,\dotsc,m\}$ where
 in the summand indexed by $\calI$, we choose $B(a)$ in the positions in $\calI$ and $[0,\be_j]$ when $j\not\in\calI$.
 A repeated application of Lemma~\ref{L:intervals} shows that this summand is
 $\MV(B(a_\calI),\dotsc,B(a_\calI))$, as projecting $a$ from the coordinates $j\not\in\calI$ gives $a_\calI$.
 This term is $|\calI|!\prod_{i\in \calI} a_i$, by the normalization property of mixed volume.\smallskip

We now show that no facial system of~\eqref{Eq:smallerSystem} has any solutions.
Since each Newton polytope is a rectangular parallelepiped $B(a)+[0,\be_j]$, its proper faces are exposed by nonzero vectors
$w\in\{-1,0,1\}^m$, and each exposes a different face.

Let $w\in\{-1,0,1\}^m$ and suppose that $w\neq 0$.
We first consider the face of $B(a)$ exposed by $w$.
This is a rectangular parallelepiped whose $i$th coordinate is
 \[
     0\mbox{ if }\ w_i=1\,,\qquad
     [0,a_i] \mbox{ if }\ w_i=0\,,\qquad\mbox{and}\qquad
     a_i\mbox{ if }\ w_i=-1\,.
 \]
 In the same manner as \eqref{def_optimal_sets}, we define $\defcolor{B(a)_w}:= \{b^*\in B(a) \mid w\cdot b^* = \min_{b\in B(a)} w\cdot b\}$,
 and we similarly define $(B(a)+[0,\be_j])_w$ for each $j=1,\ldots,m$.
 Then,
 \begin{equation}\label{Eq:boxFace}
    B(a)_w\ =\ \sum_{i\colon w_i=1}\{0\} \ +\
               \sum_{i\colon w_i=0}[0,a_i\be_i] \ +\
                  \sum_{i\colon w_i=-1}\{a_i\be_i\} \ ,
 \end{equation}
and we have
 $$(B(a)+[0,\be_j])_w = \begin{cases}
 B(a)_w, &\text{ if } w_j = 1,\\
B(a)_w+ [0,\be_j], &\text{ if } w_j = 0,\\
B(a)_w+ \be_j, &\text{ if } w_j = -1.
 \end{cases}
 $$
As $\ell_j=c_j+x_j$, we also have
 $$\ell_j(x_j)_w = \begin{cases}
c_j, &\text{ if } w_j = 1,\\
\ell_j(x_j), &\text{ if } w_j = 0,\\
x_j, &\text{ if } w_j = -1.
 \end{cases}
 $$
The Newton polytope of $f_i$ has $i$th coordinate the interval
$[0,(a_i{-}1)]$ and for~$j\neq i$ its $j$th coordinate is the interval  $[0,a_j]$.
The Newton polytope of $\ell_i\cdot f$ differs in that its $i$th coordinate is the interval
$[0,(a_i{+}1)]$.
We get
 \begin{equation}\label{Eq:facialPolys}
     (f_i+\ell_if)_w\ =\ \left\{\begin{array}{rcl}
     (f_i)_w+c_i \cdot f_w&\ &\mbox{if } w_i=1\\
     (f_i)_w+\ell_i \cdot f_w&\ &\mbox{if } w_i=0\\
     x_i\cdot f_w&\ &\mbox{if } w_i=-1
     \end{array}\right. ,
 \end{equation}
 and for $f_i$ general $(f_i)_w\neq 0$ when $w_i\neq 1$.

 Let \defcolor{$\alpha$} be the number of coordinates of $w$ equal to $0$, \defcolor{$\beta$} be the number of coordinates equal to $1$ and
 set~$\defcolor{\gamma}:=n-\alpha-\beta$, which is the number of coordinates of $w$ equal to $-1$.
 The faces of $(B(a)+[0,\be_j])_w$ exposed by~$w$ have dimension $\alpha$, by~\eqref{Eq:boxFace}, so the facial system $F_w$
 of~\eqref{Eq:smallerSystem} is effectively in $\alpha$ variables.
 Suppose first that $\gamma>0$.
 Since on $(\CC^\times)^n$ each variable $x_i$ is nonzero, by~\eqref{Eq:facialPolys} the facial system $F_w$ is equivalent to
 \[
    f_w\,,\ \{(f_i)_w\mid w_i\neq -1\}\,.
 \]
 As these are nonzero and general given their support, and there are $\alpha+\beta+1>\alpha$ of them, we see that $F_w$ has no solutions.

 If $\gamma=0$, then $\beta >0$. Consider the subfamily $\widehat{F}$ of systems of the form~\eqref{Eq:smallerSystem} where $f=0$, but the
 $f_i$ remain general.
 Then the facial system $F_w$ is equivalent to the system $\{(f_i)_w\mid w_i\neq -1\}$ of $\alpha+\beta>\alpha$ polynomials
 which are nonzero and general given their support, so that $\widehat{F}_w$ has no solutions.

 As the condition that  $F_w$ has no solutions is an open condition in the space of all systems~\eqref{Eq:pyramidPolys},
 this implies that for a general system~\eqref{Eq:pyramidPolys} with corresponding system   $F$~\eqref{Eq:smallerSystem},
 no facial system $F_w$ has a solution.
This completes the proof of the lemma.  \qed

%
%
%

\bibliographystyle{amsplain} \bibliography{EDD}

\end{document}